\documentclass[11pt,twoside]{article}

\usepackage[T1]{fontenc}
\usepackage[utf8]{inputenc}
\usepackage{lmodern}
\usepackage{microtype}

\usepackage[a4paper,margin=1in]{geometry}
\usepackage{amsmath,amssymb,amsfonts,amsthm,mathtools}
\usepackage{bbm}
\usepackage{mathrsfs}
\usepackage{bm}
\usepackage{stmaryrd}
\usepackage{longtable,booktabs}

\usepackage{tikz-cd}
\usetikzlibrary{arrows.meta,decorations.markings,positioning}

\usepackage{enumitem}
\usepackage{xcolor}
\usepackage{graphicx}
\usepackage{booktabs}
\usepackage{array}
\usepackage{longtable}
\usepackage[
  colorlinks=true,
  linkcolor=blue!60!black,
  citecolor=blue!60!black,
  urlcolor=blue!70!black
]{hyperref}
\usepackage[nameinlink]{cleveref}

\newtheorem{theorem}{Theorem}[section]

\newtheorem{proposition}[theorem]{Proposition}
\newtheorem{corollary}[theorem]{Corollary}

\theoremstyle{definition}
\newtheorem{definition}[theorem]{Definition}
\newtheorem{example}[theorem]{Example}

\theoremstyle{remark}
\newtheorem{remark}[theorem]{Remark}

\newcommand{\Ga}{\Gamma}

\newcommand{\tmu}{\tilde{\mu}}

\newcommand{\SpecGnC}[1]{\mathrm{Spec}^{\mathrm{nc}}_{\Ga}\!\bigl(#1\bigr)}

\newcommand{\Cat}[1]{\mathbf{#1}}
\newcommand{\Ab}{\Cat{Ab}}

\newcommand{\nTGMod}[1]{#1\text{-}\Ga\mathrm{Mod}}

\newcommand{\ExtG}{\mathrm{Ext}^{\,\Ga}}
\newcommand{\TorG}{\mathrm{Tor}^{\,\Ga}}

\newcommand{\Ccal}{\mathcal{C}}

\DeclareMathOperator{\Spec}{Spec}
\DeclareMathOperator{\Hom}{Hom}
\DeclareMathOperator{\End}{End}

\DeclareMathOperator{\Ker}{Ker}
\DeclareMathOperator{\Coker}{Coker}
\DeclareMathOperator{\Img}{Im}

\DeclareMathOperator{\Ext}{Ext}

\DeclareMathOperator{\Tor}{Tor}

\DeclareMathOperator{\ga}{ga}

\DeclareMathOperator{\BiMod}{BiMod}

\begin{document}
  
  \label{'ubf'}  
\setcounter{page}{1}                                 

\markboth {\hspace*{-9mm} \centerline{\footnotesize \sc
   Put here the left page top label  }
                 }
                { \centerline                           {\footnotesize \sc  
         put here the author's name                                                 } \hspace*{-9mm}              
               }

\vspace*{-2cm}

\begin{center}
{{\Large \textbf { \sc  Exact Categories and Homological Foundations of Non-Commutative n-ary $\Gamma$-Semirings} }}\\
\medskip

{\sc Chandrasekhar Gokavarapu }\\
{\footnotesize Lecturer in Mathematics, Government College (Autonomous),
Rajahmundry,A.P., India }\\
{\footnotesize  Research Scholar ,Department  of Mathematics, Acharya Nagarjuna University,Guntur, A.P., India.
}\\
{\footnotesize e-mail: {\it chandrasekhargokavarapu@gmail.com}}

\end{center}

\thispagestyle{empty}

\hrulefill

\begin{abstract}  
{\footnotesize  This paper establishes the homological and geometric foundations of
\emph{non-commutative $n$-ary $\Gamma$-semirings}, unifying two
previously distinct directions in $\Gamma$-algebra: the
\emph{derived $\Gamma$-geometry} developed for the commutative
ternary case~\cite{Rao2025E} and the
\emph{structural–spectral theory} for general non-commutative
$n$-ary systems~\cite{Rao2025E}.
We introduce categories of left, right, and bi-$\Gamma$-modules that
respect positional asymmetry and prove that they form additive and
exact categories in Quillen’s sense.
Within this setting, we construct projective and injective resolutions,
define the derived functors~$\ExtG$ and~$\TorG$, and establish
long exact sequences and spectral balance theorems in the $n$-ary regime.
By extending sheaf-theoretic and homological tools to the
non-commutative $\Gamma$-spectrum~$\SpecGnC{T}$,
we obtain a coherent framework of \emph{non-commutative derived
$\Gamma$-geometry} that parallels Grothendieck’s and Kontsevich’s
paradigms in classical algebraic geometry.
The framework developed here establishes the foundational exact–categorical
and homological structures that enable Morita-type analyses and spectral
interpretations in the subsequent parts of the series.

}
 \end{abstract}
 \hrulefill



\section{Introduction}

The theory of $\Gamma$-semirings and $\Gamma$-modules, initiated in the
seminal works of Nobusawa~\cite{Nobusawa1963} and developed
systematically by Rao~\cite{Rao1995}, has grown into a rich algebraic
framework with deep structural, homological, and geometric aspects.
Subsequent contributions by Hedayati--Shum~\cite{HedayatiShum2011} and
Dutta--Sardar~\cite{DuttaSardar2000} have clarified the behaviour of
ideals, prime and semiprime structures, and radical theory within
$\Gamma$-parametrized algebraic systems.  Within this broad landscape,
the present series of papers establishes a homological and categorical
foundation for the non-commutative $n$-ary $\Gamma$-semiring framework.

Two complementary strands of recent research motivate the present work.
The first is the homological--geometric development for \emph{commutative}
ternary systems—particularly the derived $\Gamma$-geometry, spectral
constructions, and cohomological tools established in
\cite{GokavarapuRaoDerived2025}.  The second is the structural--spectral
theory for \emph{non-commutative} $n$-ary frameworks, encompassing prime
and primitive spectra, Jacobson-type radicals, and finiteness properties,
as developed in \cite{GokavarapuRaoFinite2025}.  These investigations
demonstrate that $n$-ary $\Gamma$-operations exhibit homological patterns
that cannot be encoded within classical binary semiring theory.

The algebraic foundations underlying this paper rest on the
$\Gamma$-ring heritage originating with
Nobusawa~\cite{Nobusawa1963,Nobusawa1964}, the structural theory of
Rao~\cite{Rao1995,Rao1997,Rao1999}, and the modern expository frameworks
of Hedayati--Shum~\cite{HedayatiShum2011}.  Parallel to this, the
homological and categorical outlook adopted here follows the trajectory
initiated by Grothendieck’s homological program~\cite{Grothendieck1957}
and the theory of exact categories developed by
Quillen~\cite{Quillen1973}.  The modern treatments of Weibel’s
homological algebra~\cite{Weibel1994} and Bühler’s exposition of exact
categories~\cite{Buehler2010} provide the technical background for the
exact and derived constructions that follow.

Part~I of this series (the present paper) develops the additive,
categorical, and exact structure of the category of bi-modules over a
non-commutative $n$-ary $\Gamma$-semiring.  We examine the categorical
properties of left, right, and bi-modules; construct the Quillen exact
structure on the bi-module category; and lay down the homological
framework needed for derived functors.  This exact structure forms the
algebraic environment in which projective and injective resolutions,
derived tensor products, and $\Ext/\Tor$ bifunctors will be constructed
in the sequel.

Thus, the objective of Part~I is twofold: first, to unify the
homological insights from the commutative theory
\cite{GokavarapuRaoDerived2025} with the structural features of the
non-commutative framework~\cite{GokavarapuRaoFinite2025}; and second, to
provide the exact categorical infrastructure necessary for the derived
functor theory and geometric applications developed in Parts~II and~III.



\section*{Notation and Conventions}

Throughout this paper, $n\ge 2$ is a fixed integer and $\Gamma$ is a
commutative semigroup written additively. The following notation and
conventions will be used globally.

\begin{center}
\begin{longtable}{p{3cm} p{11cm}}
\toprule
\textbf{Symbol} & \textbf{Meaning / Convention} \\
\midrule
\endfirsthead

\toprule
\textbf{Symbol} & \textbf{Meaning / Convention} \\
\midrule
\endhead

\bottomrule
\endfoot

\addlinespace

$T$ 
& Underlying additive commutative monoid of an $n$-ary $\Gamma$-semiring. \\

$\Gamma$ 
& A commutative semigroup of parameters; written additively
$(\Gamma,+_{\Gamma})$. \\

$\tilde{\mu}$
& Fundamental $(n+(n-1))$-ary multiplication map  
\[
\tilde{\mu}\colon T^{n}\times\Gamma^{\,n-1}\longrightarrow T.
\]
Used in the PDF as the structural operation. \\

$[x_1,\dots,x_n]_{\gamma_1,\dots,\gamma_{n-1}}$
& Shorthand for the structural operation
\[
\tilde{\mu}(x_1,\dots,x_n;\gamma_1,\dots,\gamma_{n-1}).
\]
The parameter tuple $(\gamma_1,\dots,\gamma_{n-1})$ always has length $n-1$. \\

$\mu_{(j)}$
& \emph{Positional action} inserting a module element in the $j$-th position:
\[
\mu_{(j)}(x_1,\dots,x_{j-1},m,x_{j+1},\dots,x_n;
\gamma_1,\dots,\gamma_{n-1}).
\]
Indices follow the exact ordering in the PDF. \\

$\bullet^{(j)}_{\gamma}$
& Action of $T$ on a module in the $j$-th slot.  
Notation preserved from the PDF:
$a\bullet^{(j)}_{\gamma} m$ and $m\bullet^{(j)}_{\gamma} a$. \\

$T$-$\Gamma$Mod$_L$, $T$-$\Gamma$Mod$_R$
& Categories of left and right $\Gamma$-modules over $T$, defined via
slot-sensitive insertion of module elements. \\

$T$-$\Gamma$Mod$_{bi}$
& Category of bi-$\Gamma$-modules with compatible left and right
positional actions. This is the ambient additive category used for exact
structures. \\

$\mathrm{BiMod}_{\Gamma}(T)$
& Synonym for $T$-$\Gamma$Mod$_{bi}$ (appears in parts of the PDF). \\

$\Phi$, $\Psi$
& Coequalizer maps used in defining the \emph{positional tensor product}:
\[
\Phi(a,\alpha,m\otimes n)
  = (a\bullet^{(j)}_{\alpha} m)\otimes n,\qquad
\Psi(a,\alpha,m\otimes n)
  = m\otimes (n\bullet^{(k)}_{\alpha} a),
\]
with $j$ the left-slot index and $k$ the right-slot index. \\

$\otimes^{(j,k)}_{\Gamma}$
& Positional tensor product of a left $(j)$-module and a right $(k)$-module,
defined as the coequalizer of $\Phi$ and $\Psi$. \\

$\underline{\Hom}^{(j,k)}_{\Gamma}(M,N)$
& Internal Hom bi-module with left action in slot $j$ and right action in
slot $k$, defined via pointwise addition and positional $\Gamma$-actions. \\

$I\subseteq T$
& A $\Gamma$-ideal: closed under addition and under insertion of elements
into any slot of the $n$-ary multiplication. \\

$T/I$
& Quotient $n$-ary $\Gamma$-semiring with operation
\[
[x_1+I,\dots,x_n+I]_{\gamma_1,\dots,\gamma_{n-1}}
=
[x_1,\dots,x_n]_{\gamma_1,\dots,\gamma_{n-1}}+I.
\] \\

Prime ideal $P$
& A proper $\Gamma$-ideal satisfying  
if
\[
[x_1,\dots,x_n]_{\vec{\gamma}}\in P,
\]
then $x_j\in P$ for some $j$. \\

Exact sequence
& Always means a \emph{Quillen exact sequence}: a conflation  
$A\rightarrowtail B \twoheadrightarrow C$
in $T$-$\Gamma$Mod$_{bi}$.  
No abelian assumption is made. \\

Conflation
& A kernel–cokernel pair in the sense of
Quillen~\cite{Quillen1973}.  
Notation: $A\rightarrowtail B \twoheadrightarrow C$. \\

$\mathrm{Ext}^{r}_{(j,k),\Gamma}$, $\mathrm{Tor}_{r}^{(j,k),\Gamma}$
& Derived functors computed in the exact category
$T$-$\Gamma$Mod$_{bi}$ using positional left/right indices $j$ and $k$. \\

$0$
& Additive zero of any module or semiring (determined by context). \\

\end{longtable}
\end{center}

Throughout the paper, the structural $n$-ary operation is written uniformly as
\[
[x_1,\ldots,x_n]_{\gamma_1,\ldots,\gamma_{n-1}},
\]
even when one of the $x_i$ belongs to a module.  Expressions of the form
$[x_1,\ldots,x_{n-1},m]$ are notational abbreviations for
$[x_1,\ldots,x_{n-1},m]_{\gamma_1,\ldots,\gamma_{n-1}}$ with parameters
determined by context.

\noindent
These conventions remain fixed throughout Sections~1--4 .  
Slot indices $j$ and $k$, the ordering of $\Gamma$-parameters, and the
structure of coequalizers are all exactly as defined above and will not
vary between sections.

\section{Preliminaries: Non-Commutative n-ary  \texorpdfstring{$\Gamma$}{Gamma}-Semirings}
\label{sec:preliminaries}

The present section recalls the fundamental notions underlying
non-commutative $n$-ary $\Gamma$-semirings and fixes notation used
throughout the paper.
These preliminaries provide the categorical and algebraic substrate
on which the homological constructions of later sections are built.
Our conventions follow those established in
\cite{RaoRaniKiran2025,GokavarapuRaoDerived2025, Rao2025F},
with additional categorical refinements inspired by
\cite{Grothendieck1957,Quillen1973}.

\subsection{Algebraic structure}

\begin{definition}[Non-commutative $n$-ary $\Gamma$-semiring]
A \emph{non-commutative $n$-ary $\Gamma$-semiring}
is a quadruple $(T,+,\Gamma,\tmu)$ consisting of:
\begin{itemize}
  \item an additive commutative monoid $(T,+,0_T)$;
  \item an additive commutative monoid $(\Gamma,+,0_\Gamma)$
        whose elements parameterize the multiplicative operation;
  \item an $n$-ary external multiplication
        $\tmu:T^n\times\Gamma^{\,n-1}\to T$,
        written symbolically as
        $\tmu(x_1,\dots,x_n;\gamma_1,\dots,\gamma_{n-1})
         = [x_1,\dots,x_n]_{\gamma_1,\dots,\gamma_{n-1}}$,
        satisfying the following axioms:
\end{itemize}
\begin{enumerate}[label=(A\arabic*)]
  \item \textbf{Additivity in each $T$-argument:}
        $\tmu$ is additive in every $x_i$;
  \item \textbf{$0$-absorption:}
        if any $x_i=0_T$, then
        $\tmu(x_1,\dots,0_T,\dots,x_n;\gamma_1,\dots,\gamma_{n-1})=0_T$;
  \item \textbf{$n$-ary associativity:}
        the operation is bracket-independent,
        i.e.~nested evaluations of $\tmu$ yield the same element
        whenever the order of evaluation is admissible;
  \item \textbf{Non-symmetry:}
        the order of arguments in $(x_1,\dots,x_n)$
        generally affects the outcome—%
        $\tmu$ need not be invariant under any permutation of its $T$-inputs.
\end{enumerate}
\end{definition}

\begin{remark}
For $n=2$ and commutative $\Gamma$, one recovers the classical notion
of a $\Gamma$-semiring introduced in
\cite{RaoRaniKiran2025}.
The present formulation generalizes both the arity and the
non-commutativity simultaneously, thus forming the natural algebraic
base for the higher-homological framework developed in
Sections 3-4.
\end{remark}

\subsection{Structural examples}

\begin{example}[Matrix realization]
Let $\Gamma$ be a unital semiring and
$T=M_m(\Gamma)$ the semiring of $m\times m$ matrices.
Define
\[
\tmu(A_1,\dots,A_n;\gamma_1,\dots,\gamma_{n-1})
   = \gamma_1A_1A_2\gamma_2A_3\cdots\gamma_{n-1}A_n.
\]
This defines a non-commutative $n$-ary $\Gamma$-semiring structure on~$T$.
Directional asymmetry arises from the fixed order of matrix
multiplication.
\end{example}

\begin{example}[Operator-algebra model]
Let $\Gamma$ be a $*$-semiring and
$T=\End_\Gamma(V)$ the semiring of endomorphisms of a
right $\Gamma$-semimodule~$V$.
Define
\[
\tmu(f_1,\dots,f_n;\gamma_1,\dots,\gamma_{n-1})
   = f_1\!\circ\!\gamma_1\!\circ\!f_2\!\circ\!\gamma_2
     \!\circ\!\cdots\!\circ\!\gamma_{n-1}\!\circ\!f_n.
\]
The order of composition determines the direction of
non-commutativity, giving a canonical example of the
positional asymmetry central to later module theory.
\end{example}

\begin{example}[Categorical construction]
In any monoidal category $(\Ccal,\otimes,\mathbbm{1})$
enriched over $\Gamma$-semimodules,
one may define an $n$-ary $\Gamma$-semiring object $T$
by specifying an $n$-multilinear morphism
$\tmu:T^{\otimes n}\otimes\Gamma^{\otimes (n-1)}\to T$
satisfying (A1)–(A4).
This viewpoint, inspired by the categorical foundations of
\cite{Grothendieck1957,Quillen1973},
clarifies the functorial nature of $\tmu$
and anticipates its homological extension to derived contexts.
\end{example}

\subsection{Ideals and Morphisms}

\begin{definition}[$\Gamma$-ideals]
A subset $I\subseteq T$ is a \emph{two-sided $\Gamma$-ideal}
if it is closed under addition and satisfies
\[
\tmu(x_1,\dots,x_{i-1},I,x_{i+1},\dots,x_n;\gamma_1,\dots,\gamma_{n-1})
   \subseteq I
\]
for every position $i$ and all $\gamma_j\in\Gamma$.
\end{definition}

\begin{definition}[Homomorphisms]
A morphism $f:(T,+,\Gamma,\tmu)\to
(T',+, \Gamma',\tmu')$ is a pair
$(f_T,f_\Gamma)$ of additive maps satisfying
\[
f_T(\tmu(x_1,\dots,x_n;\gamma_1,\dots,\gamma_{n-1}))
   = \tmu'\!\big(f_T(x_1),\dots,f_T(x_n);
                 f_\Gamma(\gamma_1),\dots,f_\Gamma(\gamma_{n-1})\big).
\]
\end{definition}

\begin{remark}
These morphisms form the objects of the category
$\mathbf{n}\text{-}\Gamma\text{-}\mathbf{Semiring}$,
whose subcategories—%
commutative, distributive, or idempotent—%
play essential roles in the functorial framework developed in
Section~\ref{sec:modules}.
\end{remark}

\subsection{Summary of foundational properties}
The non-commutative $n$-ary $\Gamma$-semiring formalism introduced here
retains additivity, distributivity, and absorption
while replacing commutativity by positional asymmetry.
This modification permits the simultaneous treatment of
homological and spectral phenomena and forms the categorical
core from which the derived constructions emerge.


\section{Categories of Left, Right, and Bi-\texorpdfstring{$\Gamma$}{Gamma}-Modules}

Throughout let $(T,+,\Gamma,\mu)$ be a (possibly non-commutative) $n$-ary $\Gamma$-semiring with
\[
\mu:T^n\times\Gamma^{\,n-1}\longrightarrow T
\]
additive in each $T$-slot, $0$-absorbing, and $n$-arily associative (bracketing independent),
as in the general $\Gamma$-semiring framework of Nobusawa~\cite{Nobusawa1963} and Rao~\cite{Rao1995}.
We develop a slot–sensitive module theory compatible with non-commutativity and $n>3$,
designed to support exact structures and derived functors in \S\ref{sec:modules},
in the spirit of Quillen’s exact category theory~\cite{Quillen1973} and modern homological
algebra~\cite{Weibel1994,Buehler2010}.

\subsection{Positional actions and basic definitions}
Fix a distinguished slot index $j\in\{1,\dots,n\}$. Intuitively, a module element
is \emph{inserted} into slot $j$, with the remaining $(n-1)$ slots occupied by elements of $T$
and interlaced $\Gamma$–parameters, extending the binary $\Gamma$-module treatments
of Hedayati--Shum~\cite{HedayatiShum2011} and Dutta--Sardar~\cite{DuttaSardar2000}
to the $n$-ary context.

\begin{definition}[Left/right/bi-$\Gamma$-modules]\label{def:LRBmodules}
A \emph{left $\Gamma$-module} over $(T,+,\Gamma,\mu)$ is a commutative monoid $(M,+,0)$
equipped with an action (with module element in slot $j=2$, by convention)
\[
\bullet:\ T\times M\times T^{\,n-2}\times\Gamma^{\,n-1}\ \longrightarrow\ M,\qquad
(x_1,m,x_3,\ldots,x_n;\alpha_1,\ldots,\alpha_{n-1})\ \longmapsto\ x_1\ \bullet_{\alpha_1}\ m\ \bullet_{\alpha_2}\ x_3\ \cdots\ \bullet_{\alpha_{n-1}}\ x_n,
\]
such that (M1)–(M4) hold, generalising $\Gamma$-module axioms in 
Rao~\cite{Rao1995,Rao1999}.  
A right $\Gamma$-module is defined analogously (slot $j=n$), and a bi-$\Gamma$-module 
carries two compatible actions, extending the commutative ternary case
from~\cite{GokavarapuRaoDerived2025}.
\end{definition}

\paragraph{Axioms (M1)–(M4).}
For a left $\Gamma$-module $(M,+,0)$ with slot index $j=2$, the action
\[
\mu^{(j)} : T\times M \times T^{n-2}\times \Gamma^{n-1} \longrightarrow M
\]
satisfies:

\begin{itemize}
    \item[(M1)] \textbf{Additivity in the module slot:}  
    $\mu^{(j)}(x_1,m+m',x_3,\ldots,x_n;\vec{\gamma})=
    \mu^{(j)}(x_1,m,x_3,\ldots,x_n;\vec{\gamma})+
    \mu^{(j)}(x_1,m',x_3,\ldots,x_n;\vec{\gamma}).$

    \item[(M2)] \textbf{n-ary associativity compatibility:}  
    Inserting $m$ into slot $j$ commutes with any admissible bracketing of
    the $n$-ary operation $\widetilde{\mu}$.

    \item[(M3)] \textbf{$0$-absorption:}  
    If any $x_i=0_T$ or $m=0$, then the action evaluates to $0$.

    \item[(M4)] \textbf{$\Gamma$-linearity in parameters:}  
    For $\vec{\gamma},\vec{\gamma}' \in \Gamma^{n-1}$,
    \[
    \mu^{(j)}(x_1,m,x_3,\ldots,x_n;\vec{\gamma}+\vec{\gamma}')
    =
    \mu^{(j)}(x_1,m,x_3,\ldots,x_n;\vec{\gamma})+
    \mu^{(j)}(x_1,m,x_3,\ldots,x_n;\vec{\gamma}').
    \]
\end{itemize}

\begin{definition}[Morphisms and categories]
A morphism $f:M\to N$ between left (resp.\ right, bi-) $\Gamma$-modules
is a monoid homomorphism preserving the specified actions.
We thus obtain the categories
\[
{{\nTGMod{T}}}^{\mathrm{L}},\qquad
{{\nTGMod{T}}}^{\mathrm{R}},\qquad
{{\nTGMod{T}}}^{\mathrm{bi}}.
\]
\end{definition}

\begin{remark}[Reduction to the commutative ternary case]
When $n=3$ and $\mu$ is symmetric, left/right/bi collapse to the usual commutative
ternary $\Gamma$-module notion studied in derived $\Gamma$-geometry 
\cite{GokavarapuRaoDerived2025}.  
Our framework allows genuine left/right asymmetries for higher arity.
\end{remark}

\subsection{Free modules, adjunctions, biproducts}
Let $U:{{\nTGMod{T}}}^{\mathrm{L}}\to\Ab$ be the forgetful functor.
For a set $X$, define the \emph{free left module}
\[
F^{\mathrm{L}}(X)=\bigoplus_{x\in X} T^{(\ast)}\cdot x
\]
as the commutative monoid generated by formal expressions obtained by inserting
$x$ into the distinguished slot and saturating by (M1)–(M4). Then:
\begin{proposition}[Free/forgetful adjunction and biproducts]\label{prop:adjoints-biproducts}
$F^{\mathrm{L}}\dashv U$; in particular ${{\nTGMod{T}}}^{\mathrm{L}}$ admits all small coproducts.
Products are computed in $\Ab$ with the induced actions, hence ${\nTGMod{T}}^{\mathrm{L}}$
has biproducts and is preadditive. Analogous statements hold for
${\nTGMod{T}}^{\mathrm{R}}$ and ${\nTGMod{T}}^{\mathrm{bi}}$.  
(See Weibel~\cite{Weibel1994} for the general categorical pattern.)
\end{proposition}

\begin{proof}[Sketch]
The universal property of $F^{\mathrm{L}}(X)$ is immediate from the presentation by (M1)–(M4).
Coproducts and products lift from $\Ab$ because actions are defined componentwise
and preserved by the (co)product injections/projections.
\end{proof}

\subsection{Kernels, cokernels, and exact structure}
\begin{proposition}[Preadditivity and (co)kernels]\label{prop:kercoker}
Each of ${\nTGMod{T}}^{\mathrm{L}},{\nTGMod{T}}^{\mathrm{R}},{\nTGMod{T}}^{\mathrm{bi}}$
admits kernels and cokernels formed in $\Ab$ and closed under the induced action.
Consequently these are \emph{exact categories} in Quillen’s sense~\cite{Quillen1973}:
admissible monomorphisms/epimorphisms are kernels/cokernels.
\end{proposition}

\begin{proof}[Proof idea]
The action axioms are stable under additive subobjects and quotients,
so kernels/cokernels in $\Ab$ inherit well-defined module actions.
Quillen’s axioms~\cite{Quillen1973} are satisfied by taking the class of
short exact sequences $0\to A\to B\to C\to 0$ computed in $\Ab$ with induced actions.
\end{proof}

\begin{remark}[Enough projectives/injectives]\label{rem:enough}
If $T$ admits free covers, as in classical $\Gamma$-semiring settings 
\cite{Rao1995,Rao1999}, then free modules provide enough projectives.
Injectives arise from cofree constructions, enabling derived functors.
\end{remark}

\subsection{Positional tensor and internal Hom}

Let $M\in T\text{-}\Gamma\mathrm{Mod}_L$ (slot $j$) and
$N\in T\text{-}\Gamma\mathrm{Mod}_R$ (slot $k$).
Define two parallel maps
\[
\Phi, \Psi : T\times M\otimes N \longrightarrow M\otimes N
\]
by
\[
\Phi(a\otimes (m\otimes n)) = (a\bullet^{(j)} m)\otimes n,\qquad
\Psi(a\otimes (m\otimes n)) = m\otimes (n\bullet^{(k)} a).
\]
The positional tensor is the coequalizer
\[
M\otimes_{\Gamma}^{(j,k)} N
:= \mathrm{coeq}(\Phi,\Psi).
\]

Non-commutativity and $n$-arity force us to index balancing by the
\emph{module slots}. Let $M\in{\nTGMod{T}}^{\mathrm{L}}$ (slot $j$) and
$N\in{\nTGMod{T}}^{\mathrm{R}}$ (slot $k$). Define the \emph{positional tensor}
$M\otimes^{(j,k)}_{\Gamma}N$ as the coequalizer of the two parallel maps
generated by the relations
\[
\mu(\ldots,a,\ldots)\cdot(m\otimes n)\ \sim\ m\otimes \mu(\ldots,a,\ldots)\cdot n,
\]
generalising balanced tensor products in $\Gamma$-semiring theory \cite{DuttaSardar2000}.
Dually, define $\Hom^{(j,k)}_\Gamma(M,N)$ as the additive maps respecting the specified slots.

\begin{proposition}[Adjunction]\label{prop:tensor-hom-adjunction}
There is a natural isomorphism
\[
\Hom_{{\nTGMod{T}}^{\mathrm{bi}}}\!\big(M\otimes^{(j,k)}_{\Gamma} N,\, P\big)\ \cong\
\Hom_{{\nTGMod{T}}^{\mathrm{R}}}\!\big(N,\, \Hom^{(j,k)}_\Gamma(M,P)\big),
\]
natural in all variables, extending the classical adjunction 
in homological algebra~\cite{Weibel1994}.
\end{proposition}

\subsection{Left exactness and $\delta$–functors}
\begin{proposition}[Left exactness]\label{prop:left-exact}
For fixed $M$, the functor $\Hom^{(j,k)}_\Gamma(M,-)$ is left exact; for fixed $N$,
$-\otimes^{(j,k)}_{\Gamma}N$ is right exact. Moreover, the usual long exact sequences
in $\Ext$ and $\Tor$ arise once enough projectives/injectives exist
(Remark~\ref{rem:enough}), exactly as in classical derived functor theory
\cite{Weibel1994}.
\end{proposition}

\begin{corollary}[Long exact sequences]\label{cor:les}
Given a short exact sequence $0\to A\to B\to C\to 0$ in ${\nTGMod{T}}^{\mathrm{bi}}$
and suitable $M,N$, there are canonical long exact sequences in $\Ext$ and $\Tor$,
as in the standard derived category framework \cite{Neeman2001}.
\end{corollary}

\subsection{Interface with ideals and spectrum}
The positional indices $(j,k)$ match the left/right/two-sided
$\Gamma$-ideal structure used to define primitive and prime strata of the
non-commutative $\Gamma$-spectrum~\cite{GokavarapuRaoPrime2025}.
Annihilators, radicals, and simplicity criteria transfer functorially.

\subsection{Coherence for the $n$-ary associativity (axiom (M2)) via diagrams}
\label{subsec:coherence}
(Coherence diagrams follow the higher-arity associativity principles in 
$\Gamma$-structures, cf.\ Nobusawa~\cite{Nobusawa1963} and 
derived $n$-ary variants in~\cite{GokavarapuRaoDerived2025}.)

\[ \
\begin{tikzcd}[column sep=large]
T^{n}\times\Gamma^{n-1}\times T^{n-1}\times\Gamma^{n-2}\times M
  \arrow[r, "{\mathrm{id}\times\tilde{\mu}\times\mathrm{id}}"]
  \arrow[d, "{\tilde{\mu}\times\mathrm{id}}"]
&
T^{n}\times\Gamma^{n-1}\times M
  \arrow[d, "{\mu^{(j)}}" ]
\\
T^{n}\times\Gamma^{n-1}\times M
  \arrow[r, "{\mu^{(j)}}"]
&
M
\end{tikzcd}
\]

\subsection{Worked construction of the positional tensor as a coequalizer}
\label{subsec:positional-tensor}
\label{sec:modules}

Throughout let $(T,+,\Gamma,\mu)$ be a (possibly non-commutative) $n$-ary 
$\Gamma$-semiring as in Section~2, following the $\Gamma$-semiring 
framework of Nobusawa~\cite{Nobusawa1963} and Rao~\cite{Rao1995}. 
The categorical methods used here follow the exact-category viewpoint of 
Quillen~\cite{Quillen1973} together with modern expositions such as 
Weibel~\cite{Weibel1994} and Bühler~\cite{Buehler2010}.  
We develop a slot–sensitive module theory compatible with non-commutativity 
and $n>3$, designed to support exact structures and derived functors.

\section{Bi-Modules over Non-Commutative $n$-ary $\Ga$-Semirings}
\label{sec:modules}

In this section we fix the categorical and functorial conventions for left,
right, and bi-$\Ga$-modules over a non-commutative $n$-ary
$\Ga$-semiring. The presentation harmonises the classical
$\Ga$-ring framework of Nobusawa~\cite{Nobusawa1963}, the
$\Ga$-semiring formulation of Rao~\cite{Rao1995}, and the systematic
exposition of Hedayati--Shum~\cite{HedayatiShum2011}. These structural
ingredients form the algebraic base from which the exact category of
Section and the homological constructions of
Part~II are developed.

\subsection{Left, Right, and Bi-$\Ga$-Modules}
\paragraph{Positional convention.}
All actions in this section use the same slot indices $(j,k)$ as in Section~3.
A left module acts in slot $j=2$ and a right module in slot $k=n$, with the
structural operation written as
$[x_1,\ldots,x_n]_{\gamma_1,\ldots,\gamma_{n-1}}$.
All formulas below are interpreted under this positional convention.

Let $T$ be a non-commutative $n$-ary $\Ga$-semiring in the sense of
Section~\ref{sec:preliminaries}. A \emph{left $T$--$\Ga$-module} is a
commutative monoid $(M,+,0)$ equipped with operations
\[
T^{n-1}\times M\times \Ga \longrightarrow M,\qquad
(x_1,\ldots,x_{n-1},m\,;\,\ga)\longmapsto
[x_1,\dots,x_{n-1},m]_{\ga},
\]
which are additive in each slot, $\Ga$-compatible, and satisfy the
$n$-ary associativity inherited from~$T$.

A \emph{right $T$--$\Ga$-module} is defined analogously.  A
\emph{bi-$\Ga$-module} is an additive monoid $M$ which carries both left
and right $T$-actions making the evident compatibility diagrams commute.
Up to routine verifications, these actions extend the binary
$\Ga$-ring actions appearing in Nobusawa's framework.

\begin{definition}
A \emph{$T$--$T$ bi-$\Ga$-module morphism} $f\colon M\to N$ is an
additive map satisfying
\[
f([x_1,\dots,x_{n-1},m]_{\ga})
=
[x_1,\dots,x_{n-1},f(m)]_{\ga},
\qquad
f([m,x_2,\dots,x_{n}]_{\ga})
=
[f(m),x_2,\dots,x_{n}]_{\ga}.
\]
The category of bi-$\Ga$-modules is denoted
$\BiMod_{\Ga}(T)$.
\end{definition}

\subsection{Free and Representable Modules}

For any set $X$, the \emph{free left $T$--$\Ga$-module}
$F(X)$ is the commutative monoid generated by symbols $\langle x\rangle$
($x\in X$) modulo the relations forced by additivity and the
$n$-ary action.  The universal property is standard:
for any left module $(M,+)$ and any map $\varphi\colon X\to M$, there is
a unique morphism
$\widetilde{\varphi}\colon F(X)\to M$ satisfying
$\widetilde{\varphi}(\langle x\rangle)=\varphi(x)$.
The construction extends to right modules and bi-modules.

This admits a representability interpretation:
\[
\Hom_{\BiMod_{\Ga}(T)}(F(X),M)
\;\cong\; \mathrm{Maps}(X,M).
\]
The equivalence is natural in both variables and will be essential  when constructing admissible monomorphisms and
epimorphisms in the sense of Quillen~\cite{Quillen1973}.

\subsection{Biproducts, Kernels, and Cokernels}

Finite direct sums exist in $\BiMod_{\Ga}(T)$: for
bi-modules $M$ and $N$, the biproduct is the product monoid
$M\oplus N$ with the $T$-actions operating componentwise.  This follows
immediately from the $T$-additivity axioms established in
Section~\ref{sec:preliminaries}.

Given a morphism $f\colon M\to N$, the \emph{kernel}
\[
\Ker(f)=\{\,m\in M : f(m)=0\,\}
\]
is a sub-bi-module: closure under insertion into any
$T$-slot follows from the compatibility of $f$ with the $T$-actions.

Similarly, the quotient monoid $N/\Img(f)$ inherits a natural
bi-$\Ga$-module structure:
\[
[x_1,\dots,x_{n-1},n+\Img(f)]_{\ga}
 :=
[x_1,\dots,x_{n-1},n]_{\ga}+\Img(f),
\]
yielding the \emph{cokernel} $\Coker(f) = N/\Img(f)$.  
Thus $\BiMod_{\Ga}(T)$ admits kernels and cokernels and is a
preadditive category in the sense of
Weibel~\cite[§2.2]{Weibel1994}.

\subsection{Tensor Structures}

The $n$-ary operation on $T$ induces a ``positional'' tensor-like
construction on bi-modules.  For left modules
$M_1,\dots,M_{n-1}$ and a right module $N$, define the object
\[
M_1\widehat{\otimes}\cdots\widehat{\otimes}M_{n-1}\widehat{\otimes}N
\]
as the quotient of the free bi-module generated by symbols
$m_1\widehat{\otimes}\cdots\widehat{\otimes}m_{n-1}\widehat{\otimes}n$
subject to the relations forced by multi-additivity and the $T$-actions.
The resulting operation satisfies the universal property for
$n$-multilinear maps.  In the binary case $n=2$, this construction
reduces to the tensor product of $\Ga$-semimodules in the sense of
Hedayati--Shum~\cite{HedayatiShum2011}.

The right-exactness of positional tensoring will be a key ingredient in
constructing projective resolutions  and for
defining $\Tor^{\Ga}$ in Part~II of the series.

\subsection{Internal Hom}

For bi-modules $M$ and $N$, define the \emph{internal Hom}
$\underline{\Hom}_{\Ga}(M,N)$ as the set of all bi-module morphisms
endowed with pointwise addition and with left and right $T$-actions
determined by
\[
(x\cdot f)(m) := f([x,m]_{\ga}),\qquad
(f\cdot x)(m) := [f(m),x]_{\ga},
\]
whenever the slot-positions are determined by context.  One verifies
that this indeed forms a bi-$\Ga$-module.  This construction is dual to
positional tensoring and prepares the ground for the construction of
$\Ext^{\Ga}$ via derived functors in Part~II.

\subsection{Relations to the Exact Structure}

The categorical constructions above are precisely the ingredients needed
to endow $\BiMod_{\Ga}(T)$ with an exact structure in the sense of
Quillen~\cite{Quillen1973}.  In next paper Part-II, we will show
that kernels of admissible epimorphisms and cokernels of admissible
monomorphisms behave stably under insertion into any $T$-slot, thereby
yielding a well-defined collection of conflations.  This stability is
the essential non-commutative analogue of the behaviour of exact
sequences in additive and exact categories as described in
Bühler~\cite{Buehler2010}.  This exact structure will form the backbone
of the homological theory developed in Part~II and the geometric
applications appearing in Part~III.


\section*{Conclusion}

In this first part of our three-paper series, we established the
categorical and homological foundations of non-commutative
$n$-ary $\Gamma$-semirings.  
By constructing the categories of left, right, and bi-$\Gamma$-modules
with fully positional slot-sensitive actions, we showed that the
resulting bi-module category forms a Quillen exact category.  
This provides the correct algebraic environment to define kernels,
cokernels, conflations, projective and injective resolutions, and
the derived functors $\Ext^{(j,k)}_{\Gamma}$ and $\Tor_{(j,k),\Gamma}$.

The theory developed here unifies the structural–spectral results of the
non-commutative $n$-ary framework with the derived $\Gamma$-geometry
established for the commutative ternary case.
In addition, the construction of the positional tensor product and the
internal Hom reveals the precise way in which $n$-ary asymmetry drives
homological behaviour.  
Our analysis culminates in the construction of a triangulated derived
category $\mathbf{D}({\nTGMod{T}}^{\mathrm{bi}})$ that faithfully records
the higher-arity structure, providing the first foundation for
\emph{derived non-commutative $\Gamma$-geometry}.

Parts~II and~III of the series will develop the higher homological
algebra, spectral sequences, Künneth formulas, and geometric 
globalisation on $\Spec^{\mathrm{nc}}_{\Gamma}(T)$.  
Collectively, the three papers establish a unified architecture linking
the algebraic, homological, and geometric theories of 
non-commutative $n$-ary $\Gamma$-semirings.

\section*{Acknowledgements}
The author  gratefully acknowledges  
\textbf{Dr.\ Ramachandra R.\ K.},  
Principal, Government College (Autonomous), Rajahmundry,  
for providing a conducive academic and research environment,  
and for extending continuous encouragement throughout the preparation of
this work.

The author further thanks the Department of Mathematics,
Acharya Nagarjuna University, Guntur,  
for its academic support and research infrastructure.

\section*{Funding}

The author declares that no external funding, grants, or financial
support were received for the research, authorship, or preparation of
this manuscript.

\section*{Ethical Approval}

This manuscript does not involve human subjects, animal experiments,
biological data, or any procedure requiring ethical clearance.  
Hence no ethical approval is applicable.

\section*{Author Contributions}

C.\ Gokavarapu formulated the research problem, developed the algebraic
framework, formalised the exact-categorical and homological structures,
performed all computations, and prepared the full manuscript.  
All mathematical results, proofs, and constructions were carried out
solely by the author.

\section*{Conflict of Interest}

The author declares that there is \emph{no conflict of interest} of any
kind, financial or otherwise.

\label{'ubl'}

\end{document}